\documentclass[11pt]{article}

\usepackage{amsmath, amsthm, amssymb, color, enumitem, graphicx}

\usepackage[nobysame, initials]{amsrefs}

\usepackage[draft=false]{hyperref}
\hypersetup{
    colorlinks=true,        
    linkcolor=black,        
    citecolor=black,        
    filecolor=black,        
    urlcolor=black          
}

\usepackage{mathrsfs} 
\usepackage[a4paper]{geometry}

\def\itm#1{\rm ({#1})} 
\def\itmit#1{\itm{\it #1\,}} 
\def\rom{\itmit{\roman{*}}} 
\def\abc{\itmit{\alph{*}}}

\def\arab{\itm{\arabic{*}}} 
\def\endofClaim{\hfill\scalebox{.6}{$\Box$}}

\def\l{\ell}

\def\phi{\varphi}

\def\ge{\geqslant}
\def\CC{\mathbb{C}}

\def\NN{\mathbb{N}}
\def\PP{\mathbb{P}}
\def\RR{\mathbb{R}}
\def\SS{\mathbb{S}}
\def\ZZ{\mathbb{Z}}

\newtheorem{theorem}{Theorem}[section]
\newtheorem{lemma}[theorem]{Lemma}
\newtheorem{corollary}[theorem]{Corollary}
\newtheorem{prop}[theorem]{Proposition}

\newtheorem{prob}[theorem]{Problem}

\theoremstyle{definition}
\newtheorem*{definition}{Definition}
\theoremstyle{remark}

\newcommand{\oldqed}{}

\title{Ordinary hyperspheres and spherical curves}
\author{Aaron Lin\footnote{Department of Mathematics, London School of Economics and Political Science, United Kingdom.} \and Konrad Swanepoel\footnotemark[1]}
\date{}

\begin{document}

\maketitle

\begin{abstract}
An \emph{ordinary hypersphere} of a set of points in real $d$-space, where no $d+1$ points lie on a $(d-2)$-sphere or a $(d-2)$-flat, is a hypersphere (including the degenerate case of a hyperplane) that contains exactly $d+1$ points of the set.
Similarly, a \emph{$(d+2)$-point hypersphere} of such a set is one that contains exactly $d+2$ points of the set.
We find the minimum number of ordinary hyperspheres, solving the $d$-dimensional spherical analogue of the Dirac--Motzkin conjecture for $d \ge 3$.
We also find the maximum number of $(d+2)$-point hyperspheres in even dimensions, solving the $d$-dimensional spherical analogue of the orchard problem for even $d \ge 4$.
\end{abstract}

\section{Introduction}\label{sec:intro}

An \emph{ordinary line} of a planar point set is a line that contains exactly two points of the set, while a \emph{$3$-point line} is one that contains exactly three points.
Green and Tao~\cite{GT13} proved tight bounds for the minimum number of ordinary lines and maximum number of $3$-point lines spanned by a sufficiently large non-collinear planar point set, solving the so-called Dirac--Motzkin conjecture and the classical orchard problem respectively.
They also showed that any extremal configuration for the Dirac--Motzkin conjecture is contained (up to a single point) in the union of a conic and a disjoint line, while the extremal configurations for the orchard problem lie on irreducible cubic curves.

The natural extension to circles was considered in~\cite{LMMSSZ18}, where an \emph{ordinary circle} of a planar point set is a circle (including the degenerate case of a line) that contains exactly three points of the set, and a \emph{$4$-point circle} is a circle (or line) that contains exactly four points.
Tight bounds on the minimum number of ordinary circles and maximum number of $4$-point circles spanned by a finite set of points in the plane, not all on a circle or a line, were found in \cite{LMMSSZ18}, solving the analogues of the Dirac--Motzkin conjecture and orchard problem for circles.
Any extremal configuration for the Dirac--Motzkin conjecture for circles is contained in the union of two disjoint circles, while any extremal configuration for the orchard problem for circles is contained in an irreducible algebraic curve of circular degree $2$ (see Section~\ref{sec:curves} for the definition of circular degree).

Both cases above relied on structure theorems that characterised point sets spanning few ordinary lines~\cite{GT13}*{Theorem 1.5} and circles~\cite{LMMSSZ18}*{Theorem 1.5} respectively.
In particular, such point sets are mostly contained in certain low-degree algebraic curves.

The purpose of this note is to extend these results to hyperspheres in dimensions $3$ and above, and where possible, to understand the curves on which the near-extremal configurations lie.
Just as the structure theorem in \cite{LMMSSZ18} built on that of \cite{GT13}, the structure theorem for ordinary hyperspheres that we prove here (Theorem~\ref{thm:structure}) relies on one for ordinary hyperplanes from our previous paper \cite{LS18-2} (Theorem~\ref{thm:hyperplane} below).
\begin{definition}
An \emph{ordinary hypersphere} of a set of points in $\RR^d$, where no $d+1$ points lie on a $(d-2)$-sphere or a $(d-2)$-flat, is a hypersphere (including the degenerate case of a hyperplane) that contains exactly $d+1$ points of the set.

A \emph{$(d+2)$-point hypersphere} of such a set is one that contains exactly $d+2$ points of the set.
\end{definition}

On a related note, Purdy and Smith~\cite{PS10} considered ordinary spheres in $3$-space in the slightly more restricted setting of
a finite set of points with no four concyclic and no \emph{three} collinear.
We include hyperplanes as degenerate spheres because the collection of all hyperspheres and hyperplanes is closed under inversion (see Section~\ref{sec:curves}).

Theorems~\ref{thm:extremal} and~\ref{thm:extremal2} below are our main results.
We find the exact minimum number of ordinary hyperspheres in odd dimensions and prove tight asymptotic bounds in even dimensions. 
However, we were only able to prove tight asymptotic bounds on the maximum number of $(d+2)$-point hyperspheres in even dimensions.
This solves the spherical analogue of the Dirac--Motzkin conjecture for $d \ge 3$, and of the orchard problem for even $d \ge 4$.

\begin{theorem}\label{thm:extremal}
Let $d \ge 3$ and let $n \ge Cd^32^d$ for some sufficiently large absolute constant $C>0$.
Let $P$ be a set of $n$ points in $\RR^d$ where no $d+1$ points lie on a $(d-2)$-sphere or a $(d-2)$-flat.
If $P$ is not contained in a hypersphere or a hyperplane, then
the minimum number of ordinary hyperspheres spanned by $P$ is exactly $\binom{n-1}{d}$ if $d$ is odd and is
\[ \binom{n-1}{d} - O\left(d 2^{-\frac{d}{2}}\binom{n}{d/2}\right) \]
if $d$ is even.
\end{theorem}

We will show that this minimum number of ordinary hyperspheres is attained, up to inversions, by a coset of a finite subgroup of a certain type of rational normal curve or elliptic normal curve, and when $d+1$ and $n$ are coprime, by $n-1$ points in general position on a hypersphere or a hyperplane together with a point not on the hypersphere or hyperplane.
It turns out that in odd dimensions, this latter trivial example 
is the only extremal configuration, spanning exactly $\binom{n-1}{d}$ ordinary hyperspheres.
For a detailed description of the curves just mentioned together with the groups defined on them, see our paper \cite{LS18-2}.

\begin{theorem}\label{thm:extremal2}
Let $d \ge 3$ and let $n \ge Cd^32^d$ for some sufficiently large absolute constant $C>0$.
Let $P$ be a set of $n$ points in $\RR^d$ where no $d+1$ points lie on a $(d-2)$-sphere or a $(d-2)$-flat.
Then the maximum number of $(d+2)$-point hyperspheres spanned by $P$ is bounded above by
\[ \frac{1}{d+2} \binom{n-1}{d+1} + O\left(2^{-\frac{d}{2}}\binom{n}{d/2}\right). \]
\end{theorem}
We will show that when the dimension is even, the maximum value is also attained by a coset of a finite subgroup of a certain type of elliptic normal curve or rational normal curve.
Thus the bound of Theorem~\ref{thm:extremal2} is tight when $d$ is even.
On the other hand, when $d$ is odd, we do not have any lower bound that is superlinear in $n$, nor can we show an upper bound of the form $\frac{c}{d+2}\binom{n-1}{d+1}$ where $c<1$.

\begin{prob}
Determine the maximum number of $(d+2)$-point hyperspheres spanned by a set of $n$ points in $\RR^d$ where no $d+1$ points lie on a $(d-2)$-sphere or $(d-2)$-flat, for odd $d\ge 3$.
\end{prob}

While the above results include an asymptotic error term for even dimensions, the exact extremal values can be computed recursively for $n\ge Cd^32^d$ for some sufficiently large absolute constant $C>0$.
For example, if $d=4$, the minimum number of ordinary hyperspheres is
\[ 
\begin{cases}
\binom{n-1}{4} - \frac{1}{8}n^2 + \frac{1}{12} n - 1 & \text{if } n \equiv 0 \pmod{6},\\
\binom{n-1}{4} & \text{if } n \equiv 1, 5 \pmod{6},\\
\binom{n-1}{4} -\frac{1}{8}n^2 + \frac{3}{4}n - 1 & \text{if } n \equiv 2, 4 \pmod{6},\\
\binom{n-1}{4} - \frac{2}{3}n + 2 & \text{if } n \equiv 3 \pmod{6},
\end{cases}
\]
and the maximum number of $6$-point hyperspheres is
\[ 
\begin{cases}
\frac{1}{6}\binom{n-1}{5} + \frac{1}{48}n^2 - \frac{1}{72}n + \frac16 & \text{if } n \equiv 0 \pmod{6},\\
 \frac{1}{6}\binom{n-1}{5} & \text{if } n \equiv 1, 5 \pmod{6},\\
 \frac{1}{6}\binom{n-1}{5} + \frac{1}{48}n^2 - \frac{1}{8}n +\frac16 & \text{if } n \equiv 2, 4 \pmod{6},\\
 \frac{1}{6}\binom{n-1}{5} +\frac{1}{9}n -\frac{1}{3} & \text{if } n \equiv 3 \pmod{6}.
\end{cases}
\]
For fixed $d$, these values are quasipolynomials of period~$2(d+2)$ for sufficiently large~$n$.
(A function $f\colon\NN\to\ZZ$ is a \emph{quasipolynomial of period $k$} if there exist polynomials $p_0,p_1,\dots,p_{k-1}$ such that $f(n) = p_i(n)$ where $n\equiv i\pmod{k}$.)

In contrast to the $2$-dimensional situation, there are no extremal configurations on reducible curves when $d\ge 3$.
This and Theorems~\ref{thm:extremal} and~\ref{thm:extremal2} are due to the following structure theorem for sets with few ordinary hyperspheres.
We introduce spherical curves in Section~\ref{sec:curves}, before proving our theorems in Sections~\ref{sec:proof} and~\ref{sec:proof2}.

\begin{theorem}\label{thm:structure}
Let $d \ge 3$, $K > 0$, and suppose $n > C\max\{(dK)^8,d^32^dK\}$ for some sufficiently large absolute constant $C > 0$.
Let $P$ be a set of $n$ points in $\RR^d$ where no $d+1$ points lie on a $(d-2)$-sphere or a $(d-2)$-flat.
Suppose $P$ spans at most $K\binom{n}{d}$ ordinary hyperspheres.

If $d$ is odd, then all but at most $O(d2^dK)$ points of $P$ lie on a hypersphere or a hyperplane.

If $d = 2k$ is even, then up to inversions, $P$ differs in at most $O(d2^dK)$ points from a configuration of one of the following types:
\begin{enumerate}[label=\rom]
\item A subset of a hyperplane;
\item A coset $H \oplus x$ of a subgroup $H$ of a bounded $(k-1)$-spherical rational normal curve of degree $d$, for some $x$ such that $(d+2)x \in H$;
\item A coset $H \oplus x$ of a subgroup $H$ of a $k$-spherical elliptic normal curve, for some $x$ such that $(d+2)x \in H$.
\end{enumerate}
\end{theorem}

The paper \cite{LMMSSZ18} also considered the alternative, more straightforward, definition of ordinary circles as only being proper circles excluding lines, as this was how the ordinary circles problem has been formulated originally \citelist{\cite{BB94}\cite{E67}\cite{Z11}}.
However, this definition turned out to be less natural, as the class of proper circles are not invariant under inversions, and in the general approach of \cite{LMMSSZ18} the more inclusive definition had to be considered first.
Thus in this paper we also include hyperplanes as special degenerate hyperspheres, making it closed under inversions.
We leave open finding variants of Theorems~\ref{thm:extremal} and \ref{thm:extremal2} where hyperplanes are excluded (although it should not be hard to obtain analogous results as in the $2$-dimensional case).

\section{Spherical curves and inversion}\label{sec:curves}
The algebraic geometry that we will need is classical and elementary, and useful introductions can be found in \citelist{\cite{H92}\cite{Hulek}\cite{Reid}\cite{RS85}}, and also \cite{Pedoe} for inversion and stereographic projection.
While our theorems are stated over $\RR^d$, we often need to work in the larger complex projective space $\CC\PP^d$.
We use homogeneous coordinates $[x_0,x_1,\dots,x_d]$ for a point in $\CC\PP^d$, identify the affine part where $x_0\neq 0$ with $\CC^d$, and call the hyperplane defined by $x_0=0$ the \emph{hyperplane at infinity} and denote it by $\Pi_\infty$.

We denote by $\overline{S}$ the Zariski closure of a set $S \subseteq \CC\PP^d$.
For instance, the Zariski closure of the hypersphere $\SS^{d-1}\subset\CC^d$ with equation $x_1^2 + \dotsb + x_d^2 = 1$ is the projective variety $\overline{\SS^{d-1}}$ defined by the homogeneous equation $x_0^2 = x_1^2 + \dotsb + x_d^2$.
We call the intersection $\overline{\SS^{d-1}}\cap \Pi_\infty$ the \emph{imaginary sphere at infinity} and denote it by $\Sigma_\infty$.
This is a $(d-2)$-sphere on $\Pi_\infty$ and is the intersection of $\Pi_\infty$ with the Zariski closure of any hypersphere in $\CC^d$.

We will repeatedly use the following formulation of B\'ezout's theorem.
We say that a variety is \emph{pure-dimensional} if each of its irreducible components has the same dimension.
Two varieties $X$ and $Y$ in $\CC\PP^d$ intersect \emph{properly} if $\dim(X\cap Y)=\dim(X)+\dim(Y)-d$.

\begin{theorem}[B\'ezout, \cite{H92}*{Theorem 18.4}]\label{thm:bezout}
Let $X$ and $Y$ be varieties of pure dimension in $\CC\PP^d$ that intersect properly.
Then the total degree of $X\cap Y$ is equal to $\deg(X)\deg(Y)$, where multiple components are counted with multiplicity.
\end{theorem}

We define a \emph{curve} to be a pure $1$-dimensional variety in $\CC\PP^d$ that is not necessarily irreducible, but we assume that there are no multiple components (that is, the curve is \emph{reduced}).
We say that a curve is \emph{non-degenerate} if it is not contained in a hyperplane,
and \emph{real} if each of its irreducible components contains infinitely many points of $\RR\PP^d$.
Whenever we consider a curve in $\RR\PP^d$, we implicitly assume that its Zariski closure is a real curve.

We note the following simple result for later use.

\begin{lemma}\label{lem:realideal}
The homogeneous ideal of a real curve is generated by real polynomials.
\end{lemma}

\begin{proof}
Without loss of generality, the real curve $\delta\subset\CC\PP^d$ is irreducible.
Let $I$ be the homogeneous ideal of $\delta$, and consider $I = \bigoplus_e I^{(e)}$, where $I^{(e)}$ is the set of polynomials of $I$ of degree $e$.
We show that each $I^{(e)}$ can be generated by real polynomials, whence so can $I$.
A polynomial is an element of $I^{(e)}$ if and only if the hypersurface it defines contains $\delta$, which occurs if and only if the hypersurface contains more than $e\deg(\delta)$ points of $\delta$ by Theorem~\ref{thm:bezout}.
Since $\delta$ is real and contains infinitely many real points, the coefficients of each polynomial in $I^{(e)}$ satisfy a linear system of (at least) $e\deg(\delta)+1$ real equations in $\binom{d+e}{d}$ variables.
Solving this linear system then shows that $I^{(e)}$, considered as a vector space, has a basis of real polynomials.
\end{proof}

As a consequence, we obtain the following basic fact on odd-degree curves in real projective space.

\begin{lemma}\label{lem:odd}
Let $\delta$ be a non-degenerate curve of odd degree in $\RR\PP^d$.
Then any hyperplane of $\RR\PP^d$ intersects $\delta$ in at least one point of $\RR\PP^d$.
\end{lemma}

\begin{proof}
By Lemma~\ref{lem:realideal}, the homogeneous ideal of $\delta$ is generated by real polynomials.
The lemma then follows from the fact that roots of real polynomials come in complex conjugate pairs.
Since $\delta$ has odd degree, any real hyperplane thus intersects $\delta$ in at least one real point.
\end{proof}

\begin{definition}
An \emph{$\l$-spherical curve} in $\RR^d$ is a real curve in $\CC\PP^d$ that contains exactly $\l$ pairs of complex conjugate points, counted with multiplicity, on $\Sigma_\infty$.

The \emph{spherical degree} of an $\l$-spherical curve of degree $e$ is $e - \l$.
\end{definition}

Note that since lines are $0$-spherical and circles $1$-spherical, both lines and circles have spherical degree $1$.
When $d=2$, these definitions coincide with the classical definitions of $\l$-circular curve and circular degree respectively~\cites{J77, LMMSSZ18}.

An important property of the spherical degree is that it is invariant under inversion, defined next.
We first introduce stereographic projection to be the map
\[ \pi: \CC\PP^{d+1} \supset \overline{\SS^d} \setminus{\{N\}} \rightarrow \{x_{d+1} = 0\} = \CC\PP^d, \]
where $N = [1, 0, \dotsc, 0, 1]$ is the \emph{north pole} of $\overline{\SS^d}$, and $q \in \overline{\SS^d} \setminus{\{N\}}$ is mapped to the intersection point of the line $Nq$ and the hyperplane $\{x_{d+1} = 0\}$, which we identify with $\CC\PP^d$.
Denote the tangent hyperplane to $\overline{\SS^d}$ at $N$ by $\Pi_N$.
Note that its homogeneous equation is $x_{d+1}=x_0$.

It is not difficult to see that $\Pi_N$ intersects $\overline{\SS^d}$ in the cone over $\Sigma_\infty$ with vertex~$N$.
For example, when $d=2$, the intersection of the tangent plane $\Pi_N$, which has equation $x_3=x_0$, with the sphere $x_0^2=x_1^2+x_2^2+x_3^2$ is the union of the two lines $x_1\pm ix_2 =0$, which is the cone over the two circular points at infinity $[0,1,\pm i,0]$ in $\Pi_N$.
Note that these two points form the $0$-dimensional imaginary sphere at infinity.

In particular, $\pi$ maps $\Pi_N \cap \overline{\SS^d} \setminus{\{N\}}$ to $\Sigma_\infty$.
The range of $\pi$ is thus $\{x_0 \ne 0\} \cup \Sigma_\infty = \CC^d \cup \Sigma_\infty$.
Also, $\pi$ is injective on $\overline{\SS^d}\setminus\Pi_N$, and for each $y\in\Sigma_\infty$, $\pi^{-1}(y)=\ell\setminus\{N\}$, where $\ell$ is the tangent line to $\overline{\SS^d}$ at $N$ through~$y$.

Inversion in the origin $o\in\CC^d$ is defined to be the bijective map
\[ \iota_o = \pi \circ \rho \circ \pi^{-1}: \CC^d \setminus{\{o\}} \rightarrow \CC^d \setminus{\{o\}}, \]
where $\rho$ 
is the reflection 
in the hyperplane $\{x_{d+1} = 0\}$.
Inversion in an arbitrary point $r \in \CC^d$ is then defined to be the bijective map
\[ \iota_r = \tau_r \circ\iota_o \circ \tau_{-r}: \CC^d \setminus{\{r\}} \rightarrow \CC^d \setminus{\{r\}}, \]
where $\tau_r(x) = x+r$ is the translation map taking the origin to $r$.

As is well known in real space, if $V$ is a hypersphere or a hyperplane, then the inverse $\overline{\iota_r(V\setminus\{r\})}$ is again a hypersphere or a hyperplane, depending on whether $r\notin V$ or $r\in V$ respectively.
It is also easily seen that the inverse of a circle or a line is again a circle or a line.
We next show more generally that spherical degree is preserved by inversion.

\begin{prop}\label{prop:inversion1}
Let $\delta\subset\RR\PP^d$ be a real curve of spherical degree $k$.
Then $\delta' := \overline{\pi^{-1}(\overline{\delta}\setminus\Sigma_\infty)}$ is a real curve of degree $2k$ contained in $\SS^d \subset \CC\PP^{d+1}$ that intersects $\Pi_N$ in finitely many points.
\end{prop}

\begin{proof}
Let $\delta$ be $\l$-spherical of degree $e$, where $k=e-\l$.
Then the intersection of the cone over $\overline{\delta} \subset \CC\PP^d \subset \CC\PP^{d+1}$ with vertex $N$ and $\overline{\SS^d}$ is exactly the union of the curve $\delta'$ and the lines $Nx$ for each $x \in \overline{\delta} \cap \Sigma_\infty$.
By Theorem~\ref{thm:bezout}, the intersection has total degree $2e$, hence $\deg(\delta') = 2e - 2\l = 2k$.
Since $\overline{\delta}$ intersects $\Sigma_\infty$ in only finitely many points and $\pi^{-1}$ takes real points to real points, it follows that $\delta'$ is real.
Also, since $\delta'$ consists of all irreducible components of $\pi^{-1}(\overline{\delta})$ not contained in $\Pi_N$, $\delta'$ intersects $\Pi_N$ in finitely many points.
\end{proof}

\begin{prop}\label{prop:inversion2}
Let $\delta'$ be a real curve of degree $2k$ contained in $\overline{\SS^d} \subset \CC\PP^{d+1}$.
If $\delta'$ intersects $\Pi_N$ in finitely many points,
then $\delta := \overline{\pi \circ \rho(\delta')}$ is a $(k-m)$-spherical curve of degree $2k-m$, where $m \ge 0$ is the multiplicity of $\rho^{-1}(N)$ on $\delta'$.
In particular, the spherical degree of $\delta$ is $k$.
\end{prop}

\begin{proof}
Since $\pi$ is one-to-one on $\overline{\SS^d} \setminus\Pi_N$, we have $\deg(\delta) = 2k - m$.
Let $\Pi$ be a generic hyperplane in $\CC\PP^{d+1}$.
Then 
$|\delta' \cap \Pi| = 2k$.
Since $\delta'$ intersects $\Pi_N$ in finitely many points, without loss of generality, $\delta' \cap \Pi$ is disjoint from $\Pi_N$.
    Then the hypersphere $\pi(\Pi\cap\overline{\SS^d})$ intersects $\delta$ in $2k$ distinct points in $\CC^d$.
However, $|\delta \cap \pi(\Pi\cap\overline{\SS^d})| = 2(2k-m)$ by Theorem~\ref{thm:bezout}.
So $|\delta \cap \Sigma_\infty| = 2(2k-m) - 2k = 2(k-m)$, and these points must come in complex conjugate pairs as $\delta$ is real.
This means that $\delta$ is $(k-m)$-spherical, hence its spherical degree is $(2k-m) - (k-m) = k$ as claimed.
\end{proof}

We obtain the following two corollaries almost immediately.
They generalise the planar results of \cite{J77}; see also \cite{LMMSSZ18}*{Lemma~2.5}.

\begin{corollary}\label{cor:inversion}
Let $\delta$ be a real curve of spherical degree $k+1$.
The inverse of $\delta$ in a point not on $\delta$ is a $(k+1)$-spherical curve of degree $2k+2$;
the inverse of $\delta$ in a smooth point on $\delta$ is a $k$-spherical curve of degree $2k+1$;
and the inverse of $\delta$ in a double point on $\delta$ is a $(k-1)$-spherical curve of degree $2k$.
\end{corollary}
\begin{proof}
Combine Propositions~\ref{prop:inversion1} and \ref{prop:inversion2}.
\end{proof}

Note that by Lemma~\ref{lem:odd}, all real rational normal curves in $\RR^d$ are unbounded when $d$ is odd.
On the other hand, when $d$ is even, there exist bounded real rational normal curves.
In the plane they are exactly the ellipses.
The inverse of an ellipse in a point on the ellipse is an acnodal cubic, which is a cubic curve projectively equivalent to the cubic with equation $y^2=x^3-x^2$.

In $\RR\PP^d$, a singular rational curve of degree $d+1$ has exactly one singularity, and the curve is isomorphic to a planar singular cubic. When the curve is real, the singularity is an acnode, crunode, or cusp depending on whether the singularity of the real planar cubic is an acnode, crunode or cusp.
See \cite{LS18-2}*{Section~3} for more on rational normal curves, rational singular curves (in particular, rational acnodal curves), and elliptic normal curves.

\begin{corollary}\label{cor:acnodal}
Let $d=2k$.
The inverse of a bounded $(k-1)$-spherical rational normal curve in $\RR^d$ in a point on the curve is a non-degenerate $k$-spherical rational curve of degree $d+1$, with an acnode in the point of inversion and no other singularities;
the inverse of a non-degenerate $k$-spherical rational curve of degree $d+1$ in $\RR^d$, with an acnode and no other singularities, in its acnode is a bounded $(k-1)$-spherical rational normal curve.
\end{corollary}

\begin{proof}
By Propositions~\ref{prop:inversion1} and \ref{prop:inversion2}, we only have to show that a bounded rational normal curve inverts into a curve with an acnode as singularity.
If the singularity is not an acnode, then it is a crunode or a cusp, and in either case, there are real points on the curve arbitrarily close to the singularity.
Then the inverse of this curve will be unbounded, contrary to assumption.
\end{proof}

Finally, we describe group laws on certain spherical curves that describe when points lie on a hypersphere.
Note that a real $k$-spherical elliptic normal curve has a unique real point at infinity.

\begin{prop}\label{prop:group}
Let $d = 2k$.
A bounded $(k-1)$-spherical rational normal curve or $k$-spherical elliptic normal curve in $\RR\PP^d$ has a group structure such that $d+2$ points (not including the real point at infinity on an elliptic normal curve) lie on a hypersphere if and only if they sum to the identity.
In the prior case, this group is isomorphic to $\RR/\ZZ$;
in the latter case, this group is isomorphic to $\RR/\ZZ$ or $\RR/\ZZ \times \ZZ_2$ depending on whether it has one or two semi-algebraically connected components.
\end{prop}

\begin{proof}
Let $\delta \subset \RR\PP^d \subset \CC\PP^d$ be a curve of spherical degree $k+1$.
Then by Proposition~\ref{prop:inversion1}, $\delta' := \overline{\pi^{-1}(\overline{\delta})}$ is a curve in $\CC\PP^{d+1}$ of degree $2(k+1) = d+2$.
If $\delta$ is a $(k-1)$-spherical bounded rational normal curve, then by Corollary~\ref{cor:acnodal}, $\delta'$ is a rational acnodal curve in $\CC\PP^{d+1}$.
If $\delta$ is a $k$-spherical elliptic normal curve, then $\delta'$ is an elliptic normal curve in $\CC\PP^{d+1}$.
In both cases, $\delta' \cap \RR\PP^d$ has a group structure such that $d+2$ real points on $\delta'$ lie on a hyperplane if and only if they sum to the identity, and this group is isomorphic to $\RR/\ZZ$ when $\delta'$ is acnodal, and isomorphic to $\RR/\ZZ$ or $\RR/\ZZ \times \ZZ_2$ if $\delta'$ is elliptic, depending on whether it has one or two semi-algebraically connected components \cite{LS18-2}*{Propositions~3.9 and~3.1}.
Since a generic hyperplane intersects $\overline{\SS^d}$ in a $(d-1)$-sphere and stereographic projection takes hyperspheres (not containing the north pole) to hyperspheres in $\RR\PP^d$, $\pi$ transfers the group to $\delta$ in such a way that $d+2$ points on $\delta$ lie on a hypersphere if and only if they sum to the identity.
\end{proof}

\section{Proof of Theorem~\ref{thm:structure}}\label{sec:proof}
Since ordinary hyperspheres spanned by $P \subset \RR^d$ are in one-to-one correspondence with ordinary hyperplanes in $\pi^{-1}(P) \subset \RR^{d+1}$, 
Theorem~\ref{thm:structure} is a simple consequence of the following structure theorem for sets with few ordinary hyperplanes \cite{LS18-2}, combined with the results from Section~\ref{sec:curves}.

\begin{theorem}[\cite{LS18-2}]\label{thm:hyperplane}
Let $d \ge 4$, $K > 0$, and suppose $n \ge C\max\{(dK)^8, d^32^dK\}$ for some sufficiently large absolute constant $C > 0$.
Let $P$ be a set of $n$ points in $\RR\PP^d$ where every $d$ points span a hyperplane.
If $P$ spans at most $K\binom{n-1}{d-1}$ ordinary hyperplanes, then $P$ differs in at most $O(d2^dK)$ points from a configuration of one of the following types:
\begin{enumerate}[label=\rom]
\item A subset of a hyperplane;\label{case:hyperplane}
\item A coset $H \oplus x$ of a subgroup $H$ of an elliptic normal curve or the smooth points of a rational acnodal curve of degree $d+1$, for some $x$ such that $(d+1)x \in H$.\label{case:curve}
\end{enumerate}
\end{theorem}

\begin{proof}[Proof of Theorem~\ref{thm:structure}]
Projecting $P$ stereographically, we obtain a set $P' := \pi^{-1}(P) \subset \overline{\SS^d} \subset \RR\PP^{d+1}$ of $n$ points, no $d+1$ of which lie on a hyperplane, spanning at most $K\binom{n}{d}$ ordinary hyperplanes. 
We can thus apply Theorem~\ref{thm:hyperplane} to get that $P'$ differs in at most $O(d2^dK)$ points from
\begin{enumerate}[label=\arab]
\item a subset of a hyperplane, or \label{case:1}
\item a coset $H \oplus x$ of a subgroup $H$ of an elliptic normal curve or the smooth points of a rational acnodal curve of degree $d+2$, for some $x$ such that $(d+2)x \in H$. \label{case:2}
\end{enumerate}
In the latter case, since $P'$ is contained in a hypersphere, by Theorem~\ref{thm:bezout}, the degree $d+2$ curve is also contained in the hypersphere.
In particular, it is bounded.

Suppose $d$ is odd (so that $d+2$ is odd).
Then Case~\ref{case:2} does not occur, as Lemma~\ref{lem:odd} implies an odd degree curve is always unbounded.
Thus projecting Case~\ref{case:1} back down to $\RR^d$, we get that $P$ differs in at most $O(d2^dK)$ points from a subset of a hypersphere or a hyperplane.

Now suppose $d$ is even.
If we are in Case~\ref{case:1}, then we are in the same situation as the odd case.
So assume we are in Case~\ref{case:2}, and all but at most $O(d2^dK)$ points of $P'$ lie on a degree $d+2$ curve $\delta' \subset \RR^{d+1}$, which is either elliptic or acnodal.
Projecting (stereographically) back to $\RR^d$, we get that all but at most $O(d2^dK)$ points of $P$ lie on a curve $\delta \subset \RR^d$.
By Proposition~\ref{prop:inversion2}, we get one of the following cases, depending on the multiplicity of the north pole $N$ on $\delta'$:
\begin{enumerate}[label=\abc]
\item $N$ is a double point of $\delta'$, which means $\delta'$ is acnodal, and thus $\delta$ is a bounded $(k-1)$-spherical rational normal curve of degree $d$;\label{case:a}
\item $N$ is a smooth point of $\delta'$, in which case $\delta$ is a $k$-spherical elliptic normal curve or rational acnodal curve of degree $d+1$;\label{case:b}
\item $N$ does not lie on $\delta'$, in which case $\delta$ is a $(k+1)$-spherical curve of degree $d+2$.\label{case:c}
\end{enumerate}
Note that the group structure mentioned in the statement of the theorem is inherited from that in Theorem~\ref{thm:hyperplane}, and is detailed in Proposition~\ref{prop:group}.
By Corollary~\ref{cor:inversion}, the curve from \ref{case:c} can be inverted to a curve as in \ref{case:b}, and by Corollary~\ref{cor:acnodal}, the rational curve from \ref{case:b} can be inverted to a curve as in \ref{case:a}.
\end{proof}

\section{Proof of Theorems~\ref{thm:extremal} and \ref{thm:extremal2}}\label{sec:proof2}
Similar to the previous section, Theorems~\ref{thm:extremal} and \ref{thm:extremal2} follow from the following extremal results for ordinary hyperplanes and $(d+1)$-point hyperplanes from \cite{LS18-2}.

\begin{theorem}\label{thm:hyperplane2}
Let $d \ge 4$ and let $n\ge Cd^3 2^d$ for some sufficiently large absolute constant $C > 0$.
The minimum number of ordinary hyperplanes spanned by a set of $n$ points in $\RR\PP^d$, not contained in a hyperplane and where every $d$ points span a hyperplane, is
\[ \binom{n-1}{d-1} - O\left(d 2^{-\frac{d}{2}}\binom{n}{\lfloor \frac{d-1}{2} \rfloor}\right). \]
This minimum is attained by a coset of a subgroup of an elliptic normal curve or the smooth points of a rational acnodal curve of degree $d+1$, and when $d+1$ and $n$ are coprime, by $n-1$ points in a hyperplane together with a point not in the hyperplane.
\end{theorem}

\begin{theorem}\label{thm:hyperplane3}
Let $d \ge 4$ and let $n\ge Cd^3 2^d$ for some sufficiently large absolute constant $C > 0$.
The maximum number of $(d+1)$-point hyperplanes spanned by a set of $n$ points in $\RR\PP^d$ where every $d$ points span a hyperplane is
\[ \frac{1}{d+1} \binom{n-1}{d} + O\left(2^{-\frac{d}{2}}\binom{n}{\lfloor \frac{d-1}{2} \rfloor}\right). \]
This maximum is attained by a coset of a subgroup of an elliptic normal curve or the smooth points of a rational acnodal curve of degree $d+1$.
\end{theorem}

\begin{proof}[Proof of Theorem~\ref{thm:extremal}]
This theorem follows from Theorem~\ref{thm:hyperplane2} and stereographic projection, similar to the proof of Theorem~\ref{thm:structure}.
\end{proof}

\begin{proof}[Proof of Theorem~\ref{thm:extremal2}]
This theorem follows from Theorem~\ref{thm:hyperplane3} and stereographic projection, similar to the proof of Theorem~\ref{thm:structure}.
Note that in the case when $d$ is odd, the extremal configurations in Theorem~\ref{thm:hyperplane3} lying on an algebraic curve of degree $d+2$ cannot exist, as this curve has to lie on $\overline{\SS^d}$, hence is bounded, which contradicts Lemma~\ref{lem:odd}.
\end{proof}



\begin{bibdiv}
\begin{biblist}

\bib{BB94}{article}{
      author={B{\'a}lintov{\'a}, A.},
      author={B{\'a}lint, V.},
       title={On the number of circles determined by $n$ points in the
  {Euclidean} plane},
        date={1994},
     journal={Acta Math.\ Hungar.},
      volume={63},
      number={3},
       pages={283\ndash 289},
}

\bib{E67}{article}{
      author={Elliott, P. D. T.~A.},
       title={On the number of circles determined by $n$ points},
        date={1967},
     journal={Acta Math.\ Acad.\ Sci.\ Hungar.},
      volume={18},
       pages={181\ndash 188},
}

\bib{GT13}{article}{
   author={Green, Ben},
   author={Tao, Terence},
   title={On sets defining few ordinary lines},
   journal={Discrete Comput.\ Geom.},
   volume={50},
   date={2013},
   number={2},
   pages={409--468},
}

\bib{H92}{book}{
   author={Harris, J.},
   title={Algebraic Geometry: A First Course},
   publisher={Springer-Verlag},
   date={1992},
}

\bib{Hulek}{book}{
    author={Hulek, K.},
    title={Elementary Algebraic Geometry},
    publisher={American Mathematical Society},
    date={2000},
}

\bib{J77}{article}{
      author={Johnson, W.~W.},
       title={Classification of plane curves with reference to inversion},
        date={1877},
     journal={The Analyst},
      volume={4},
      number={2},
       pages={42\ndash 47},
}

\bib{LMMSSZ18}{article}{
   author={Lin, Aaron},
   author={Makhul, Mehdi},
   author={Mojarrad, Hossein Nassajian},
   author={Schicho, Josef},
   author={Swanepoel, Konrad},
   author={de Zeeuw, Frank},
   title={On Sets Defining Few Ordinary Circles},
   journal={Discrete Comput. Geom.},
   volume={59},
   date={2018},
   number={1},
   pages={59--87},
}

\bib{LS18-2}{article}{
   author={Lin, Aaron},
   author={Swanepoel, Konrad},
   title={On sets defining few ordinary hyperplanes},
   journal={Discrete Analysis},
   volume={4},
   date={2021},
   pages={34pp.},
}

\bib{Pedoe}{book}{
   author={Pedoe, Dan},
   title={Geometry, a comprehensive course},
   edition={2},
   publisher={Dover Publications, Inc., New York},
   date={1988},
}

\bib{PS10}{article}{
   author={Purdy, George B.},
   author={Smith, Justin W.},
   title={Lines, circles, planes and spheres},
   journal={Discrete Comput. Geom.},
   volume={44},
   date={2010},
   number={4},
   pages={860--882},
}

\bib{Reid}{book}{
   author={Reid, M.},
   title={Undergraduate algebraic geometry},
   publisher={Cambridge University Press, Cambridge},
   date={1988},
}

\bib{RS85}{book}{
   author={Semple, J. G.},
   author={Roth, L.},
   title={Introduction to algebraic geometry},
   note={Reprint of the 1949 original},
   publisher={The Clarendon Press, Oxford University Press, New York},
   date={1985},
}

\bib{Z11}{article}{
      author={Zhang, R.},
       title={On the number of ordinary circles determined by $n$ points},
        date={2011},
     journal={Discrete Comput.\ Geom.},
      volume={46},
      number={2},
       pages={205\ndash 211},
}

\end{biblist}
\end{bibdiv}
\end{document}